\documentclass[11pt]{amsart}

\usepackage[T1]{fontenc}
\usepackage[utf8]{inputenc}
\usepackage{lmodern}
\usepackage[margin=1in]{geometry}
\usepackage{amsmath,amssymb,amsthm,mathtools}
\usepackage{enumitem}
\usepackage{microtype}
\usepackage[hidelinks]{hyperref}

\allowdisplaybreaks[2]
\setlist{topsep=2pt plus 1pt minus 1pt,itemsep=1pt plus 1pt,parsep=0pt,partopsep=0pt}
\makeatletter
\def\thm@space@setup{%
  \thm@preskip=6pt plus 2pt minus 2pt
  \thm@postskip=6pt plus 2pt minus 2pt
}
\makeatother

\hypersetup{
  pdftitle={On an Iwaniec--Kovalev--Onninen Conjecture for Harmonic Quasiconformal Annulus Mappings},
  pdfauthor={David Kalaj, Jinsong Liu, Jian-feng Zhu},
  pdfsubject={Harmonic quasiconformal annulus mappings}
}

\theoremstyle{plain}
\newtheorem{theorem}{Theorem}[section]
\newtheorem{lemma}[theorem]{Lemma}
\newtheorem{conjecture}[theorem]{Conjecture}
\newtheorem{corollary}[theorem]{Corollary}
\newtheorem{problem}[theorem]{Problem}

\theoremstyle{definition}
\newtheorem{example}[theorem]{Example}

\theoremstyle{remark}
\newtheorem{remark}[theorem]{Remark}
\newtheorem*{remark*}{Remark}

\newcommand{\C}{\mathbb{C}}
\newcommand{\Mod}{\operatorname{Mod}}

\newcommand{\onto}{\xrightarrow{\mathrm{onto}}}
\newcommand{\norm}[1]{\left\lVert #1 \right\rVert}

\makeatletter
\@namedef{subjclassname@2020}{\textup{2020} Mathematics Subject Classification}
\expandafter\let\expandafter\subjclassname\csname subjclassname@2020\endcsname
\makeatother

\title[On an Iwaniec--Kovalev--Onninen Conjecture]{On an Iwaniec--Kovalev--Onninen Conjecture\\
for harmonic quasiconformal annulus mappings}

\author[D. Kalaj]{David Kalaj}
\address{Faculty of Natural Sciences and Mathematics, University of Montenegro, Podgorica, Montenegro}
\email{davidkalaj@gmail.com}

\author[J. Liu]{Jinsong Liu}
\address{State Key Laboratory of Mathematical Sciences, AMSS, Chinese Academy of Sciences, Beijing 100190, China\\
School of Mathematical Sciences, University of Chinese Academy of Sciences, Beijing 100049, China}
\email{liujsong@math.ac.cn}

\author[J.-F. Zhu]{Jian-feng Zhu}
\address{Department of Mathematics, Shantou University, Shantou, Guangdong 515063, China}
\email{flandy@stu.edu.cn}

\date{May 2026}

\begin{document}
\flushbottom

\begin{abstract}
In 2012, Iwaniec, Kovalev and Onninen proposed an upper Nitsche--Gr\"otzsch type estimate for harmonic
$K$-quasiconformal homeomorphisms between circular annuli: in normalized form, every such map
$h:A(1,s)\to A(1,S)$ should satisfy
\[
S \leq \frac{K+1}{2}s-\frac{K-1}{2s}.
\]
The radial and spiral-radial one-mode models explain why this estimate is natural.

We show that this one-mode evidence does not extend to the unrestricted non-radial class. For
every $1<s<S$, we construct a non-radial harmonic orientation-preserving diffeomorphism
$h:A(1,s)\to A(1,S)$ with
\[
\norm{\omega_h}_{\infty}<\frac{S-s}{S-s^{-1}},
\]
strictly below the conjectural threshold. Thus, the dilatation lower bound predicted by the
conjecture fails, even for smooth harmonic diffeomorphisms with prescribed circular boundary
components. The construction uses a small high-frequency reparametrization of the outer
boundary: it lowers the first Fourier-mode dilatation by order $t^2$, while the compensating
high modes are exponentially damped at the inner boundary.

We also record structured regimes in which the one-mode estimate survives, including an
inner-boundary anti-conformal energy condition, a Fourier leakage criterion, and a low-frequency
spectral stability result. In the minimal-surface interpretation, the upper radial model is
helicoidal and vertical-periodic rather than single-valued catenoidal; accordingly, the
counterexamples give vertical-periodic minimal annuli with slope below the expected helicoidal
threshold, and lead to a non-radial extremal problem for harmonic annulus diffeomorphisms.
\end{abstract}


\maketitle

\section{Background and motivation}

The study of harmonic homeomorphisms between doubly connected planar domains sits at the
intersection of two classical distortion theories. On one side, quasiconformal mappings of annuli
obey Gr\"otzsch-type modulus distortion bounds. On the other side, harmonic homeomorphisms
between annuli obey Nitsche-type restrictions: a harmonic homeomorphism cannot map a round
annulus onto an arbitrarily thin round annulus. The work of Iwaniec, Kovalev and Onninen
connects these two themes by proving sharp estimates for harmonic quasiconformal mappings
between circular annuli, and by relating the extremal maps to doubly connected minimal surfaces
\cite{IKO2012}.
General background on planar harmonic mappings and their dilatations can be found in
\cite{ClunieSheilSmall1984,Duren2004,Lewy1936}; for quasiconformal mappings and annular
modulus distortion we refer to \cite{Ahlfors2006,AIM2009,Grotzsch1928,LehtoVirtanen1973}.
The annular Nitsche problem and its variants are treated in
\cite{AIM2010,IKO2010,IKO2011,Kalaj2005,Kalaj2016,Lyzzaik1999,Lyzzaik2001,Nitsche1962,Weitsman2001}.
For the minimal-surface background surrounding the Bj\"orling interpretation and catenoidal
model examples, see \cite{Dierkes1992,Nitsche1975,Osserman1986}.

For $0<r<R<\infty$, write
\[
A(r,R)=\{z\in\C:r<|z|<R\},\qquad
\Mod A(r,R)=\log \frac{R}{r}.
\]
Let
\[
h:A(r,R)\onto A(r^*,R^*)
\]
be an orientation-preserving harmonic homeomorphism. If $h$ is $K$-quasiconformal, then
\[
\norm{Dh(z)}^2\leq KJ(z,h), \quad \text{for a.e. } z,
\]
or equivalently
\[
|h_{\bar z}|\leq k |h_z|,\qquad k=\frac{K-1}{K+1}<1.
\]
Throughout the paper, a complex-valued harmonic map is written in the form
\[
h=f+\overline g,
\]
where $f$ and $g$ are holomorphic. Thus
\[
h_z=f',\qquad h_{\bar z}=\overline{g'},
\]
and the second complex dilatation is the holomorphic quotient
\[
\omega_h=\frac{g'}{f'},\qquad |\omega_h|=\frac{|h_{\bar z}|}{|h_z|}
\]
wherever $f'\ne0$. We shall always apply the maximum principle to this holomorphic second
dilatation, not to the anti-holomorphic quotient $h_{\bar z}/h_z$ itself.
The lower Nitsche--Gr\"otzsch estimate proved in \cite{IKO2012} says that, if
$h:A(r,R)\to A(r^*,R^*)$ is a $K$-quasiconformal harmonic homeomorphism, then
\begin{equation}\label{eq:lower-ng}
\frac{R^*}{r^*}\geq \frac{K+1}{2K}\frac{R}{r}+\frac{K-1}{2K}\frac{r}{R}.
\end{equation}
This estimate simultaneously refines the usual lower Gr\"otzsch distortion bound and Nitsche's
lower bound for harmonic annulus homeomorphisms. The extremal map for \eqref{eq:lower-ng} is radial:
\[
z\mapsto r^\ast\left(\frac{K+1}{2K}\frac{z}{r}+\frac{K-1}{2K}\frac{r}{\bar z}\right).
\]
Motivated by the sharpness and symmetry of this lower estimate, Iwaniec, Kovalev and Onninen
proposed the following upper analogue.

\begin{conjecture}[\cite{IKO2012}, Conjecture~1.8]
\label{conj:iko}
Let
\[
h:A(r,R)\onto A(r^*,R^*)
\]
be a $K$-quasiconformal harmonic homeomorphism. Then
\[
\frac{R^*}{r^*}
\leq \frac{K+1}{2}\frac{R}{r}-\frac{K-1}{2}\frac{r}{R}
\leq \left(\frac{R}{r}\right)^K.
\]
Equality in the first inequality is expected only for the radial map
\[
h(z)=r^*\left(\frac{K+1}{2}\frac{z}{r}-\frac{K-1}{2}\frac{r}{\bar z}\right),
\]
up to rotations and conformal automorphisms of the source annulus.
\end{conjecture}

We use the conjecture in precisely this unrestricted form: no radiality, one-mode, principal-map,
or zero-period minimal-surface hypothesis is imposed.

After scaling source and target, we shall use the normalized notation
\[
A=A(1,s),\qquad A^*=A(1,S),\qquad s>1,\quad S>1.
\]
Then Conjecture~\ref{conj:iko} becomes
\begin{equation}\label{eq:normalized-upper}
S\leq \Phi_K(s),\qquad
\Phi_K(s)=\frac{K+1}{2}s-\frac{K-1}{2s}.
\end{equation}
In terms of
\[
k=\frac{K-1}{K+1},
\]
that is
\begin{equation}\label{eq:k-form-upper}
S\leq \frac{s-k/s}{1-k}.
\end{equation}

If $S>s$, then \eqref{eq:k-form-upper} is equivalent to the lower bound
\begin{equation}\label{eq:kappa-bound}
k\geq \kappa(s,S),\qquad
\kappa(s,S):=\frac{S-s}{S-s^{-1}}.
\end{equation}
Indeed,
\[
S\leq \frac{s-k/s}{1-k}
\Longleftrightarrow (1-k)S\leq s-\frac{k}{s}
\Longleftrightarrow S-s\leq k\left(S-\frac{1}{s}\right).
\]
Thus, in the range $S>s$, the conjecture says that, every harmonic homeomorphism
$h:A(1,s)\to A(1,S)$ should satisfy the sharp dilatation lower bound
\begin{equation}\label{eq:dilatation-lower}
\norm{\omega_h}_{\infty}\geq \frac{S-s}{S-s^{-1}}.
\end{equation}
The radial map realizes equality. This is the key point of the framework adopted in this paper:
the conjecture is exactly what the canonical one-mode examples suggest. The proposed equality
case in \cite{IKO2012} is radial, and the principal harmonic maps and equality cases appearing
there are also one-mode maps. Thus, the conjecture was natural precisely, because every such
model points to the same expression.

The main point of this article is that this evidence does not control arbitrary non-radial harmonic
annulus maps. The bound \eqref{eq:dilatation-lower} is true in several structured classes, but
false in the full non-radial class. The counterexamples do not change the target annulus; the
only freedom used is the boundary correspondence on the outer circle.

The failure of the radial bound suggests a more natural extremal problem. For $1<s<S$, let
$\mathcal H(s,S)$ denote the class of harmonic orientation-preserving diffeomorphisms from
$A(1,s)$ onto $A(1,S)$. Put
\[
\mathcal H_{qc}(s,S)=\{h\in\mathcal H(s,S):\|\omega_h\|_\infty<1\}
\]
and define
\[
q_*(s,S)=\inf_{h\in\mathcal H_{qc}(s,S)}\|\omega_h\|_\infty,
\qquad
K_*(s,S)=\frac{1+q_*(s,S)}{1-q_*(s,S)}.
\]
The radial upper model gives the value $\kappa(s,S)$, whereas
Theorem~\ref{thm:counterexamples} gives $q_*(s,S)<\kappa(s,S)$. Thus, any genuine harmonic
quasiconformal extremal attaining $q_*(s,S)$, if such an extremal exists, must be non-radial.
Describing such extremal maps and their minimal-surface lifts is the natural replacement for the
original radial conjecture.

\begin{remark}[The range $S\leq s$]
The upper estimate \eqref{eq:normalized-upper} is automatic when $S\leq s$, because
\[
\Phi_K(s)-s=\frac{K-1}{2}\left(s-\frac{1}{s}\right)\geq 0.
\]
However, for $S<s$ harmonic diffeomorphisms do not exist for all parameters. The Nitsche
conjecture, proved by Iwaniec--Kovalev--Onninen \cite{IKO2011}, says that a harmonic
homeomorphism $A(1,s)\to A(1,S)$ can exist only if
\[
S\geq \frac12\left(s+\frac{1}{s}\right),
\]
and this bound is sharp. Hence the genuinely restrictive part of \cite[Conjecture 1.8]{IKO2012}
lies in the thickening range $1<s<S$.
\end{remark}

We close the introduction by describing the organization of the paper.
Section~\ref{sec2} recalls the radial equality model and fixes the normalization
against which the later constructions are measured. Section~\ref{sec3} proves the
main negative result by constructing, for each pair \(1<s<S\), a high-frequency
non-radial harmonic diffeomorphism whose maximal complex dilatation is strictly
below the conjectural threshold. Section~\ref{sec4} isolates several regimes, in which
the one-mode mechanism still forces the conjectured estimate, including an
inner-boundary energy criterion, a Fourier spectral criterion, and a
low-frequency stability theorem. Section~\ref{sec:minimal} develops the
minimal-surface interpretation: the lift criterion, the slope formula, the
catenoidal/helicoidal sign distinction, the vertical-periodic lift of the
counterexamples, and a perturbative zero-period theorem for single-valued minimal
graphs. Section~\ref{sec:conclusion} summarizes the mechanism behind the failure, formulates the
resulting non-radial extremal problem, and separates the planar, periodic, and zero-period
problems. Appendix~\ref{Apx} records the
numerical details for the explicit example with \((s,S)=(2,3)\).

\section{The known radial model case}\label{sec2}

We first recall the radial model case, which is the normalized equality candidate in
\cite[Conjecture 1.8]{IKO2012}. The purpose of this section is only to fix the
calibration for later comparison.

\begin{lemma}
\label{lem:radial}
Let
$h:A(1,s)\onto A(1,S)$
be a radial orientation-preserving $K$-quasiconformal harmonic homeomorphism. Then
\begin{equation}\label{eq:radial-model-bound}
S\leq \frac{K+1}{2}s-\frac{K-1}{2s}.
\end{equation}
Equality holds, up to rotation, precisely for
\begin{equation}\label{eq:radial-extremal}
h(z)=\frac{K+1}{2}z-\frac{K-1}{2}\frac{1}{\bar z}.
\end{equation}
\end{lemma}

\begin{proof}
After a rotation of the target, a radial degree-one harmonic map has the form
\[
h(\rho e^{i\theta})=H(\rho)e^{i\theta},\qquad H(1)=1,\quad H(s)=S.
\]
The harmonic equation gives
\[
H''(\rho)+\frac{1}{\rho}H'(\rho)-\frac{1}{\rho^2}H(\rho)=0,
\]
thus, $H(\rho)=a\rho+b/\rho.$
Equivalently,
\[
h(z)=az+\frac{b}{\bar z}.
\]
The inner and outer boundary conditions are
$a+b=1$ and $as+b/s=S.$
Moreover,
$h_z=a$ and $h_{\bar z}=-b/\bar{z}^2.$
Since the map is orientation preserving, we have $a>0$. The $K$-quasiconformal condition is therefore
\[
\frac{|b|}{a|z|^2}\leq k,
\qquad
k=\frac{K-1}{K+1}.
\]
The maximum occurs on $|z|=1$, because $|z|^{-2}$ decreases on $1\le |z|\le s$. Hence
\begin{equation}\label{eq:radial-b-bound}
|b|\leq ka.
\end{equation}
If $b\geq0$, then $a=1-b$ and
\[
S=(1-b)s+\frac{b}{s}=s-b\left(s-\frac{1}{s}\right)\leq s\leq \Phi_K(s).
\]
If $b<0$, write $b=-\beta$, $\beta>0$. Then $a=1+\beta$, and \eqref{eq:radial-b-bound} gives
\[
\beta\leq k(1+\beta),
\qquad
\beta\leq \frac{k}{1-k}=\frac{K-1}{2}.
\]
Consequently
\[
S=(1+\beta)s-\frac{\beta}{s}
=s+\beta\left(s-\frac{1}{s}\right)
\leq s+\frac{K-1}{2}\left(s-\frac{1}{s}\right),
\]
which is exactly \eqref{eq:radial-model-bound}. Equality requires $\beta=(K-1)/2$ and $a=(K+1)/2$, which gives \eqref{eq:radial-extremal}.
\end{proof}

\begin{remark}
For the equality candidate
\[
H_K(z)=\frac{K+1}{2}z-\frac{K-1}{2\bar z},
\]
one has
\[
H_K(\rho e^{i\theta})=\left(\frac{K+1}{2}\rho-
\frac{K-1}{2\rho}\right)e^{i\theta},
\]
then $|H_K(e^{i\theta})|=1$ and $|H_K(se^{i\theta})|=\Phi_K(s)$. Moreover, one has
\[
\frac{|(H_K)_{\bar z}|}{|(H_K)_z|}
=\frac{K-1}{K+1}\frac{1}{|z|^2}
\leq \frac{K-1}{K+1}.
\]
Thus, the constant in the radial model case is sharp. This is the normalized equality
candidate stated with Conjecture~1.8 in \cite{IKO2012}, so Lemma~\ref{lem:radial} should be
read as a calibration rather than as a new theorem.
\end{remark}

\section{A general family of non-radial counterexamples}\label{sec3}
\label{sec:counterexamples}

We now show that Conjecture~\ref{conj:iko} fails in the full non-radial class. The construction is useful
because it exposes the exact mechanism of failure: we lower the radial first-mode dilatation by
decreasing the first Fourier coefficient $J_0(t)$, while the high-frequency modes needed to keep
the outer radius equal to $S$ are almost invisible on the inner boundary.

At the heuristic level, the outer boundary map
$e^{i\theta}\mapsto e^{i(\theta+t\sin m\theta)}$ changes the first Fourier coefficient from
$S$ to $SJ_0(t)=S-St^2/4+O(t^4)$. This reduces the radial
anti-conformal-to-conformal ratio by a fixed multiple of $t^2$. The compensating modes have
frequencies $1+mj$, $j\ne0$, and their trace at the inner circle is suppressed by the factor
$s^{-|1+mj|}$. Choosing $m\asymp \log(1/t)$ makes this leakage $o(t^2)$ and leaves a strict
margin below the threshold $\kappa(s,S)$.

We shall use two elementary lemmas to keep the analytic and topological parts of the construction
separate.

\begin{lemma}\label{lem:dilatation-maximum}
Let $h=f+\overline g$ be harmonic in $A(1,s)$ and $C^1$ on the closed annulus, with $f$ and $g$
holomorphic. Put
\[
P(z)=2zf'(z),\qquad G(z)=2zg'(z).
\]
If $P$ has no zeros on the closed annulus, then $\omega_h=g'/f'=G/P$
is holomorphic in $A(1,s)$ and continuous on its closure. Hence
\[
\|\omega_h\|_\infty=\max_{\partial A(1,s)}|\omega_h|.
\]
Moreover, on each circle $|z|=\rho$,
\[
|\omega_h|
=
\frac{|G|}{|P|}
=
\frac{|\rho h_\rho+i h_\theta|}{|\rho h_\rho-i h_\theta|}.
\]
\end{lemma}

\begin{proof}
The convention $h=f+\overline g$ gives $h_z=f'$ and $h_{\bar z}=\overline{g'}$. Also,
\[
2zh_z=\rho h_\rho-i h_\theta=P,
\qquad
2\bar z h_{\bar z}=\rho h_\rho+i h_\theta=\overline{G}.
\]
Thus $|h_{\bar z}|/|h_z|=|G|/|P|$; the displayed polar formula is the same identity with
$\rho h_\rho+i h_\theta=\overline{G(z)}$. Since $P$ is non-vanishing, $G/P$ is holomorphic in
the annulus and continuous up to the boundary. The maximum principle on the annulus gives the
asserted boundary maximum.
\end{proof}

\begin{lemma}\label{lem:annulus-degree}
Let $h$ be continuous on $\overline{A(1,s)}$ and $C^1$ in $A(1,s)$. Assume that $J_h>0$ in the
annulus, that $h$ maps $|z|=1$ homeomorphically onto $|w|=1$ with degree one, and that $h$ maps
$|z|=s$ homeomorphically onto $|w|=S$ with degree one. Then
\[
h:A(1,s)\to A(1,S)
\]
is a diffeomorphism, provided the two boundary components have the standard annular orientation.
\end{lemma}

\begin{proof}
Let $A=A(1,s)$. For $w\notin h(\partial A)$, the Brouwer degree of $h$ at $w$ is the winding number
of the image of $\partial A$ around $w$, with the outer boundary positively oriented and the inner
boundary negatively oriented. If $1<|w|<S$, the outer image circle winds once around $w$ and the
inner image circle does not wind around $w$, so $\deg(h,A,w)=1$. If $|w|<1$, the two boundary
circles both wind once around $w$ but with opposite orientations, so the degree is zero. If
$|w|>S$, both winding numbers are zero.

Since $J_h>0$, every interior preimage contributes local degree $+1$. Hence every point of
$A(1,S)$ has exactly one preimage, while no point in $\C\setminus\overline{A(1,S)}$ has a preimage.
No interior point can map to either boundary circle: otherwise the openness of the local
diffeomorphism would force nearby image points on both sides of that circle, including points whose
degree is zero. Thus $h(A(1,s))=A(1,S)$ and $h$ is one-to-one. The inverse function theorem then
shows that $h$ is a diffeomorphism.
\end{proof}

\begin{theorem}
\label{thm:counterexamples}
Fix $1<s<S$ and set
\[
\kappa=\kappa(s,S)=\frac{S-s}{S-s^{-1}}.
\]
Then, there exists a non-radial harmonic orientation-preserving diffeomorphism $h:A(1,s)\rightarrow A(1,S)$
such that
\[
\norm{\omega_h}_{\infty}<\kappa.
\]
Consequently, the conjectured upper estimate \eqref{eq:normalized-upper} is false in the full non-radial class for every
pair $1<s<S$.
\end{theorem}

\begin{proof}
The proof has three parts. First, we compute the strict $t^2$ decrease in the first Fourier mode.
Second, we show that the high-frequency modes contribute only $o(t^2)$ on the inner boundary and
remain harmless on the outer boundary. Finally, we use the boundary maximum principle for
$\omega_h$ and the degree lemma to obtain a global diffeomorphism.

Let $0<t<1$ be small, and let $m\in\mathbb{N}$, $m\geq 2$.
The precise choice of $m=m(t)$ will be made below. Prescribe
\begin{equation}\label{eq:counterexample-boundary-data}
h(e^{i\theta})=e^{i\theta},
\qquad
h(se^{i\theta})=Se^{i(\theta+t\sin m\theta)}.
\end{equation}
If $mt<1$, then the outer boundary map is orientation preserving because
\[
\frac{\mathrm{d}}{\mathrm{d}\theta}(\theta+t\sin m\theta)
=
1+mt\cos m\theta
\geq 1-mt>0.
\]
By the Jacobi--Anger expansion \cite[Chapter~3]{Watson1944} (or
\cite[p.~31]{ColtonKress1992}), we obtain
\[
e^{it\sin m\theta}
=
\sum_{j\in\mathbb{Z}}J_j(t)e^{ijm\theta},
\]
where $J_j$ is the Bessel function of the first kind. Hence
\begin{equation}\label{eq:outer-fourier}
Se^{i(\theta+t\sin m\theta)}
=
\sum_{j\in\mathbb{Z}}c_j e^{in_j\theta},
\qquad
c_j=SJ_j(t),\quad n_j=1+mj.
\end{equation}
For $j\neq0$, put
\[
p_j=|n_j|,
\qquad
V_p(\rho)=\frac{\rho^p-\rho^{-p}}{s^p-s^{-p}}.
\]
Then $V_p(1)=0$ and $V_p(s)=1$. The first Fourier mode is written as
\[
\left(a\rho+\frac{b}{\rho}\right)e^{i\theta},
\]
where
$a+b=1$ and $as+b/s=SJ_0(t)$.
Thus
\[
a(t)=\frac{SJ_0(t)-s^{-1}}{s-s^{-1}},
\qquad
b(t)=\frac{s-SJ_0(t)}{s-s^{-1}}.
\]
Since $S>s$ and $J_0(t)\to1$ as $t\to0$, for all sufficiently small
$t>0$ we have $SJ_0(t)>s$. Hence, $b(t)<0$. Set
\[
\beta(t)=-b(t)
=
\frac{SJ_0(t)-s}{s-s^{-1}}>0.
\]
Since harmonic functions on an annulus separate into Fourier modes, and since
the non-first modes vanish on the inner boundary and have boundary coefficient
$c_j$ on $|z|=s$, their radial factors are
\[
V_{p_j}(\rho)
=
\frac{\rho^{p_j}-\rho^{-p_j}}{s^{p_j}-s^{-p_j}},
\qquad
p_j=|n_j|.
\]
Thus, the harmonic extension of \eqref{eq:counterexample-boundary-data} is
\[
h(\rho e^{i\theta})
=
\left(a\rho-\frac{\beta}{\rho}\right)e^{i\theta}
+
\sum_{j\neq0}c_jV_{p_j}(\rho)e^{in_j\theta}.
\]
The series converges absolutely and uniformly with all derivatives on
$1\leq \rho\leq s$, because the Bessel coefficients decay factorially.

Let $h=f+\overline g$ and set
\[
P(z)=2zh_z(z)=2zf'(z),
\qquad
G(z)=2zg'(z).
\]
In polar coordinates,
\[
P=\rho h_\rho-ih_\theta,
\qquad
2\bar z h_{\bar z}=\rho h_\rho+ih_\theta.
\]
Since $2\bar z h_{\bar z}=\overline{G(z)}$, we have
\[
|\omega_h(z)|
=
\frac{|g'(z)|}{|f'(z)|}
=
\frac{|G(z)|}{|P(z)|}
=
\frac{|\rho h_\rho+ih_\theta|}
{|\rho h_\rho-ih_\theta|}.
\]

For the first mode, we have
\[
P_0(\rho e^{i\theta})=2a\rho e^{i\theta},
\qquad
(\rho h_\rho+ih_\theta)_0=\frac{2\beta}{\rho}e^{i\theta}.
\]
Thus, the first-mode complex dilatation is
\[
\omega_0(z,t)
=
\frac{g_0'(z,t)}{f_0'(z,t)}
=
\frac{\beta(t)}{a(t)}\frac{1}{z^2}.
\]
We write
\[
q_0(t):=\frac{\beta(t)}{a(t)}
=
\frac{SJ_0(t)-s}{SJ_0(t)-s^{-1}}.
\]
Then $|\omega_0(\rho e^{i\theta},t)|=q_0(t)/\rho^2$,
and therefore
\[
\sup_{1\leq \rho\leq s}
|\omega_0(\rho e^{i\theta},t)|
=
q_0(t).
\]

Using
\[
J_0(t)=1-\frac{t^2}{4}+O(t^4),
\]
we obtain
\begin{equation}\label{eq:q0-expansion}
q_0(t)
=
\kappa
-
\frac{S(s-s^{-1})}{4(S-s^{-1})^2}t^2
+
O(t^4),
\end{equation}
where
\[
\kappa=q_0(0)=\frac{S-s}{S-s^{-1}}.
\]
Hence, the first Fourier mode alone has
\[
\sup_{1\leq \rho\leq s}
|\omega_0(\rho e^{i\theta},t)|
=
q_0(t)<\kappa
\]
for all sufficiently small $t>0$.

It remains to show that the non-first modes do not destroy this strict
inequality. For a single non-first mode $V_p(\rho)e^{in\theta}$, $p=|n|$,
its contribution to $P=\rho h_\rho-ih_\theta$
is $\left(\rho V_p'(\rho)+nV_p(\rho)\right)e^{in\theta}$
and its contribution to
$Q:=\rho h_\rho+ih_\theta$
is $\left(\rho V_p'(\rho)-nV_p(\rho)\right)e^{in\theta}$.
Write
\[
P=P_0+P_{\mathrm{pert}},
\qquad
Q=Q_0+Q_{\mathrm{pert}}.
\]
For the first mode, one has $P_0(\rho e^{i\theta})=2a\rho e^{i\theta}$ and $Q_0(\rho e^{i\theta})=\frac{2\beta}{\rho}e^{i\theta}$.
For \(j\neq0\), the \(n_j\)-th Fourier component of the harmonic extension is
\[
c_j V_{p_j}(\rho)e^{i n_j\theta},\qquad p_j=|n_j|.
\]
its contribution to $P$ is
\[c_j\left(\rho V_{p_j}'(\rho)+n_jV_{p_j}(\rho)\right)
e^{in_j\theta}\]
and its contribution to $Q$ is
\[c_j \left(\rho V_{p_j}'(\rho)-n_jV_{p_j}(\rho)\right)
e^{in_j\theta}.\]
On the inner boundary $\rho=1$, we have $V_{p_j}(1)=0$ and
\[
V_{p_j}'(1)
=
\frac{2p_j}{s^{p_j}-s^{-p_j}}.
\]
Therefore, by the triangle inequality, we get
\[
|P_{\mathrm{pert}}(e^{i\theta})|
\leq
\sum_{j\neq0}|c_j|
\frac{2p_j}{s^{p_j}-s^{-p_j}}
=:E_1(t,m),
\]
and similarly
\[
|Q_{\mathrm{pert}}(e^{i\theta})|
\leq
E_1(t,m).
\]
Consequently, on $|z|=1$,
\[
|Q(e^{i\theta})|
\leq
2\beta(t)+E_1(t,m),
\]
while
\[
|P(e^{i\theta})|
\geq
2a(t)-E_1(t,m).
\]

We now estimate $E_1(t,m)$ as follows. For $0<t\leq1$ and $j\geq1$,
\[
|J_j(t)|
\leq
e^{1/4}\frac{(t/2)^j}{j!},
\qquad
|J_{-j}(t)|=|J_j(t)|.
\]
Indeed, the power-series expansion gives
\[
J_j(t)
=
\sum_{\ell=0}^{\infty}
\frac{(-1)^\ell}{\ell!(j+\ell)!}
\left(\frac{t}{2}\right)^{j+2\ell}.
\]
Hence
\[
|J_j(t)|
\leq
\frac{(t/2)^j}{j!}
\sum_{\ell=0}^{\infty}
\frac{1}{\ell!}
\left(\frac{t^2}{4}\right)^\ell
=
e^{t^2/4}
\frac{(t/2)^j}{j!}
\leq
e^{1/4}
\frac{(t/2)^j}{j!}.
\]
Also, since $J_{-j}(t)=(-1)^jJ_j(t)$, we have $|J_{-j}(t)|=|J_j(t)|$.

Now $p_j=|1+mj|$. For $m\geq2$ and $j\neq0$,
\[
p_j\geq m|j|-1,
\qquad
p_j\leq m|j|+1\leq 2m|j|.
\]
Moreover, for $p\geq1$,
\[
s^p-s^{-p}
=
s^p(1-s^{-2p})
\geq
s^p(1-s^{-2}).
\]
Therefore
\[
\frac{p_j}{s^{p_j}-s^{-p_j}}
\leq
C_s p_j s^{-p_j}
\leq
C_s m|j|s^{-m|j|}.
\]
Since $c_j=SJ_j(t)$, it follows that
\[
E_1(t,m)
\leq
C_{s,S}m
\sum_{j\neq0}
|j|\,|J_j(t)|s^{-m|j|}.
\]
Using the Bessel-coefficient bound above, we get
\[
E_1(t,m)
\leq
C_{s,S}m
\sum_{k=1}^{\infty}
k\frac{(t/2)^k}{k!}s^{-mk}.
\]
Putting $x=t/(2s^m)$ and using
\[
\sum_{k=1}^{\infty}k\frac{x^k}{k!}=xe^x,
\]
we obtain
\[
E_1(t,m)
\leq
C_{s,S}mxe^x
\leq
C_{s,S}mt s^{-m}.
\]

Now, we choose
\begin{equation}\label{eq:m-choice}
m=m(t)
=
\left\lceil
\frac{3\log(1/t)}{\log s}
\right\rceil.
\end{equation}
Then $m(t)\to\infty$ as $t\to0$, but slowly enough that, for sufficiently small $t>0$,
$m(t)\geq2$ and $m(t)t<1$. Moreover, $m(t)t\to0$ and
\[
m(t)t s^{-m(t)}
=
O\bigl(t^4\log(1/t)\bigr)
=
o(t^2).
\]
Indeed, from $m(t)=\left\lceil 3\log(1/t)/\log s\right\rceil$, we have
$m(t)=O(\log(1/t))$ and $m(t)\geq 3\log(1/t)/\log s$. Hence
\[
s^{-m(t)}\leq s^{-3\log(1/t)/\log s}=t^3.
\]
Therefore
\[
m(t)t s^{-m(t)}
=O(\log(1/t))\,t\,O(t^3)
=O\bigl(t^4\log(1/t)\bigr)=o(t^2).
\]
This implies that
\begin{equation}\label{eq:E1-small}
E_1(t,m(t))=o(t^2).
\end{equation}
Let
\[
c_{s,S}=\frac{S(s-s^{-1})}{4(S-s^{-1})^2}>0
\]
be the coefficient of $t^2$ in \eqref{eq:q0-expansion}. It follows from \eqref{eq:q0-expansion}, after
shrinking $t$ if necessary, that
\[
q_0(t)\leq \kappa-\frac{c_{s,S}}{2}t^2.
\]
Also $a(t)\to (S-s^{-1})/(s-s^{-1})>0$. Since $E_1(t,m(t))=o(t^2)$, and since the quotient
\[
(a,\beta,E)\mapsto \frac{2\beta+E}{2a-E}
\]
is uniformly Lipschitz near $(a(0),\beta(0),0)$, we may further shrink $t$ so that
\[
\frac{2\beta(t)+E_1(t,m(t))}
{2a(t)-E_1(t,m(t))}
\leq
q_0(t)+\frac{c_{s,S}}{4}t^2
\leq
\kappa-\frac{c_{s,S}}{4}t^2
<\kappa.
\]
Therefore
\begin{equation}\label{eq:inner-omega-bound}
\sup_{|z|=1}|\omega_h(z)|<\kappa.
\end{equation}

We next estimate the outer boundary. Put
\[
\lambda_p
=
sV_p'(s)
=
p\frac{s^p+s^{-p}}{s^p-s^{-p}}.
\]
At $\rho=s$, the non-first-mode contributions to $P$ and $Q$ are bounded by
\[
E_s^P(t,m)
=
\sum_{j\neq0}
|c_j|\,|\lambda_{p_j}+n_j|,
\qquad
E_s^Q(t,m)
=
\sum_{j\neq0}
|c_j|\,|\lambda_{p_j}-n_j|.
\]
Since
\[
\lambda_p
=
p\frac{1+s^{-2p}}{1-s^{-2p}}
\leq C_s p,
\]
and since $|n_j|=p_j$, we have $|\lambda_{p_j}+n_j|+|\lambda_{p_j}-n_j|\leq C_s p_j$.
Using $p_j\leq2m|j|$ for $m\geq2$, we get
\[
E_s^P(t,m)+E_s^Q(t,m)
\leq
C_{s,S}m
\sum_{j\neq0}|j|\,|J_j(t)|.
\]
By the same Bessel-coefficient bound, one has
\[
\sum_{j\neq0}|j|\,|J_j(t)|
\leq
C
\sum_{k=1}^{\infty}
k\frac{(t/2)^k}{k!}
=
C\frac{t}{2}e^{t/2}
\leq
Ct.
\]
Therefore
\begin{equation}\label{eq:outer-error-bound}
E_s^P(t,m)+E_s^Q(t,m)
\leq
C_{s,S}mt.
\end{equation}
For the choice \eqref{eq:m-choice}, this tends to zero.

On $|z|=s$, the first-mode terms satisfy
\[
|Q_0(se^{i\theta})|
=
\frac{2\beta(t)}{s},
\qquad
|P_0(se^{i\theta})|
=
2a(t)s.
\]
Thus
\[
|\omega_h(se^{i\theta})|
\leq
\frac{2\beta(t)/s+E_s^Q(t,m(t))}
{2a(t)s-E_s^P(t,m(t))}.
\]
The denominator is positive for all sufficiently small $t$, because $a(t)\to a(0)>0$ and
$E_s^P(t,m(t))=O(m(t)t)\to0$. The right-hand side tends, as $t\to0$, to
\[
\frac{\beta(0)}{a(0)s^2}
=
\frac{\kappa}{s^2}
<\kappa.
\]
Thus, the strict margin at the limiting outer boundary persists for small $t$. Therefore
\begin{equation}\label{eq:outer-omega-bound}
\sup_{|z|=s}|\omega_h(z)|<\kappa
\end{equation}
for all sufficiently small $t>0$.

It remains to prove that $\omega_h$ is holomorphic in the annulus. The first
mode contributes
$P_0(\rho e^{i\theta})=2a\rho e^{i\theta}$, and thus
\[
|P_0|\geq2a.
\]
For $1\leq\rho\leq s$,
\[
\rho V_p'(\rho)+pV_p(\rho)
=
\frac{2p\rho^p}{s^p-s^{-p}},
\qquad
\rho V_p'(\rho)-pV_p(\rho)
=
\frac{2p\rho^{-p}}{s^p-s^{-p}}.
\]
Consequently, whether $n_j=p_j$ or $n_j=-p_j$, we have
\[
\sup_{1\leq\rho\leq s}
\left|
\rho V_{p_j}'(\rho)+n_jV_{p_j}(\rho)
\right|
\leq
C_s p_j.
\]
Therefore
\[
\sum_{j\neq0}|c_j|
\sup_{1\leq\rho\leq s}
\left|
\rho V_{p_j}'(\rho)+n_jV_{p_j}(\rho)
\right|
\leq
C_{s,S}
\sum_{j\neq0}p_j|J_j(t)|.
\]
Using $p_j\leq2m|j|$ and the Bessel-coefficient bound again, we obtain
\begin{equation}\label{eq:P-perturb-bound}
\sum_{j\neq0}|c_j|
\sup_{1\leq\rho\leq s}
\left|
\rho V_{p_j}'(\rho)+n_jV_{p_j}(\rho)
\right|
\leq
C_{s,S}mt.
\end{equation}
Since
\[
a(t)\to \frac{S-s^{-1}}{s-s^{-1}}>0
\qquad\text{and}\qquad
m(t)t\to0,
\]
we may choose $t>0$ so small that the right-hand side of
\eqref{eq:P-perturb-bound}, with $m=m(t)$, is smaller than $a(t)$.
Hence, on the closed annulus, we get
\[
|P|
\geq
|P_0|-|P-P_0|
\geq
2a(t)-a(t)
=
a(t)>0.
\]
This shows that $P$ has no zeros on the closed annulus. Lemma~\ref{lem:dilatation-maximum},
together with \eqref{eq:inner-omega-bound} and \eqref{eq:outer-omega-bound}, gives
\[
\norm{\omega_h}_{\infty}<\kappa.
\]
In particular,
\[
J_h=|h_z|^2-|h_{\bar z}|^2>0,
\]
so $h$ is an orientation-preserving local diffeomorphism. The boundary maps in
\eqref{eq:counterexample-boundary-data} have degree one on their respective target circles, and
therefore Lemma~\ref{lem:annulus-degree} shows that $h$ is a diffeomorphism of $A(1,s)$ onto
$A(1,S)$. The map is non-radial, because the outer boundary parametrization contains the
nonconstant angular term $t\sin m\theta$.

Finally, if the conjectured upper estimate were true for this map with
\[
K=\frac{1+q}{1-q},
\qquad
q=\norm{\omega_h}_{\infty},
\]
then \eqref{eq:kappa-bound} would force $q\geq\kappa$, contradicting
$q<\kappa$.
\end{proof}

\begin{remark}
For $(s,S)=(2,3)$, the threshold is
\[
\kappa(2,3)=\frac{3-2}{3-1/2}=\frac25.
\]
Taking $t=1/50$ and $m=20$ in the above construction gives a harmonic orientation-preserving
diffeomorphism $h:A(1,2)\to A(1,3)$ with
\[
\norm{\omega_h}_{\infty}<0.39993<\frac25.
\]
The details are recorded in the appendix. They make the contradiction explicit: for
$q=0.39993$, the corresponding quasiconformal constant is
\[
K=\frac{1+q}{1-q}=2.33294\ldots<\frac73,
\]
and the conjectured right-hand side equals
\[
\frac{K+1}{2}\,2-\frac{K-1}{4}=2.999708\ldots,
\]
so the conjectured inequality would force $3\leq2.999708\ldots$.
\end{remark}

\begin{corollary}\label{cor:nonradial-extremal-value}
For $1<s<S$, let $\mathcal H(s,S)$ denote the class of harmonic orientation-preserving
diffeomorphisms from $A(1,s)$ onto $A(1,S)$, and put
\[
\mathcal H_{qc}(s,S)=\{h\in\mathcal H(s,S):\|\omega_h\|_\infty<1\}.
\]
Set

\[
q_*(s,S)=\inf_{h\in\mathcal H_{qc}(s,S)}\|\omega_h\|_\infty,
\qquad
K_*(s,S)=\frac{1+q_*(s,S)}{1-q_*(s,S)}.
\]
Then
\[
q_*(s,S)<\kappa(s,S)=\frac{S-s}{S-s^{-1}}.
\]
The radial upper model has value exactly $\kappa(s,S)$. Hence any minimizer attaining
$q_*(s,S)$, should such a minimizer exist, is necessarily non-radial.
\end{corollary}

\begin{proof}
The radial upper model
\[
h_{\rm rad}(z)=a_0z-\frac{\beta_0}{\bar z},
\qquad
 a_0=\frac{S-s^{-1}}{s-s^{-1}},
\qquad
 \beta_0=\frac{S-s}{s-s^{-1}},
\]
maps $A(1,s)$ onto $A(1,S)$ and satisfies
\[
\|\omega_{h_{\rm rad}}\|_\infty=\frac{\beta_0}{a_0}=\kappa(s,S).
\]
Theorem~\ref{thm:counterexamples} gives an admissible non-radial harmonic diffeomorphism with
strictly smaller maximal dilatation. This proves the strict inequality for $q_*(s,S)$ and rules out
radial extremals.
\end{proof}

\begin{remark}\label{rem:extremal-attainment}
The proof of Theorem~\ref{thm:counterexamples} does not require the infimum defining
$q_*(s,S)$ to be attained. A compactness approach for normalized quasiconformal homeomorphisms
suggests a possible existence theorem for extremals, but a complete proof would have to control the
boundary correspondence, rule out collapse of boundary components, and preserve the degree in the
limit. We therefore treat attainment and uniqueness of extremals as part of the open extremal
problem stated in Section~\ref{sec:conclusion}.
\end{remark}

\section{The one-mode framework and positive criteria}\label{sec4}

The counterexample identifies the obstruction to the full conjecture. We now isolate several
hypotheses under which the one-mode mechanism still controls the map and the conjectured
estimate remains valid. These results are meant as supporting criteria explaining the range of the
one-mode intuition.

We begin by recording the elementary complex-coefficient one-mode calculation. The genuinely
new positive statements in this section are the energy, Fourier, and low-frequency criteria that
follow.

\subsection{The known one-mode spiral-radial model}

We next recall the standard one-mode, or spiral-radial, model. This is the elementary
complex-coefficient version of the same calculation underlying the principal harmonic maps in
\cite{IKO2012}. The maps have the form
\[
h(z)=az+\frac{b}{\bar z},
\qquad a,b\in\C.
\]
If $b/a$ is not real, then the image of each circle $|z|=\rho$ is still a circle centered at the
origin, but the angular rotation of this circle depends on $\rho$. Thus the mapping is generally
not radial, even though it has only one angular Fourier mode.

\begin{lemma}
\label{lem:onemode}
Let
\[
h(z)=az+\frac{b}{\bar z},
\qquad a,b\in\C,
\]
be an orientation-preserving harmonic diffeomorphism $h:A(1,s)\rightarrow A(1,S)$.
Assume the boundary radii are normalized by $|a+b|=1$ and $|as+b/s|=S$.
Set $q=\norm{\omega_h}_{\infty}=|b|/|a|<1$.
Then
\[
S\leq \frac{s-q/s}{1-q}.
\]
Consequently, if $h$ is $K$-quasiconformal, then
\begin{equation}\label{eq:onemode-K-bound}
S\leq \frac{K+1}{2}s-\frac{K-1}{2s}.
\end{equation}
Equality in \eqref{eq:onemode-K-bound} forces $b/a=-(K-1)/(K+1)$ and recovers the radial extremal map, up to
rotation.
\end{lemma}

\begin{proof}
We have $h_z=a$ and $h_{\bar z}=-b/\bar{z}^2$.
Hence
\[
|\omega_h(z)|=\frac{|b|}{|a|\,|z|^2},
\qquad
q=\norm{\omega_h}_{\infty}=\frac{|b|}{|a|}.
\]
Put $c=b/a$ and $|c|=q$.
The boundary normalization gives $|a|=1/|1+c|$,
and therefore
\begin{equation}\label{eq:S-complex-coeff}
S=\left|as+\frac{b}{s}\right|
=\frac{|s+c/s|}{|1+c|}.
\end{equation}
Write $c=qe^{i\varphi}$ and $x=\cos\varphi$. Squaring \eqref{eq:S-complex-coeff}, we have
\[
S^2=\frac{s^2+q^2/s^2+2qx}{1+q^2+2qx}.
\]
As a function of $x\in[-1,1]$, the right-hand side has derivative
\[
\frac{2q(1+q^2-s^2-q^2/s^2)}{(1+q^2+2qx)^2}<0,
\]
because $s>1$. Thus, it is maximized at $x=-1$, i.e.\ at $c=-q$. Hence
\[
S\leq \frac{s-q/s}{1-q}.
\]
If $h$ is $K$-quasiconformal, then $q\leq k=(K-1)/(K+1)$. Since
\[
\frac{\mathrm{d}}{\mathrm{d}q}\left(\frac{s-q/s}{1-q}\right)
=\frac{s-s^{-1}}{(1-q)^2}>0,
\]
we get
\[
S\leq \frac{s-q/s}{1-q}
\leq \frac{s-k/s}{1-k}
=\frac{K+1}{2}s-\frac{K-1}{2s}.
\]
Equality in the $K$-bound requires $q=k$ and $c=-q$. After a rotation this is exactly
\[
h(z)=\frac{K+1}{2}z-\frac{K-1}{2\bar z}.
\]
The proof of Lemma~\ref{lem:onemode} is complete.
\end{proof}

\begin{remark*}
Lemma~\ref{lem:radial} and Lemma~\ref{lem:onemode} are essentially known model computations. The radial extremal is
already the equality candidate in Conjecture~1.8 of Iwaniec--Kovalev--Onninen, and the same
one-mode structure appears in their principal harmonic maps and equality cases
$c(z+\lambda/\bar z)$ \cite{IKO2012}. The point of including the two lemmas is not to claim novelty, but to make explicit the
one-mode framework behind the conjecture. In both cases the same sharp expression
\[
\frac{K+1}{2}s-\frac{K-1}{2s}
\]
is forced, with equality only in the radial extremal case.
\end{remark*}

\subsection{An inner-boundary anti-conformal energy condition}

The next criterion is elementary but very useful. It says that the conjectured estimate follows as
soon as the anti-conformal component has enough $L^2$ energy on the inner boundary.

For a harmonic map $h=f+\overline g$, define two functions $P$, $Q$ on $|z|=1$, as follows
\[
P(z)=2zh_z(z),
\qquad
Q(z)=2\bar z h_{\bar z}(z).
\]
Then $|\omega_h(z)|=|Q|/|P|$, on $|z|=1$.

\begin{theorem}
\label{thm:energy}
Let $S>s$ and set
\[
\kappa=\frac{S-s}{S-s^{-1}}.
\]
Let $h:A(1,s)\rightarrow A(1,S)$
be a harmonic orientation-preserving diffeomorphism. Suppose that
\begin{equation}\label{eq:energy-condition}
\frac{1}{2\pi}\int_0^{2\pi}|Q(e^{i\theta})|^2\,\mathrm{d}\theta
\geq \kappa^2
\frac{1}{2\pi}\int_0^{2\pi}|P(e^{i\theta})|^2\,\mathrm{d}\theta.
\end{equation}
Then $\norm{\omega_h}_{\infty}\geq \kappa$.
Consequently, the conjectured upper estimate \eqref{eq:normalized-upper} holds for every $K$-quasiconformal map
satisfying \eqref{eq:energy-condition}.
\end{theorem}

\begin{proof}
On $|z|=1$,
\[
|Q|=|\omega_h|\,|P|\leq \norm{\omega_h}_{\infty}|P|.
\]
Squaring and integrating gives
\[
\frac{1}{2\pi}\int_0^{2\pi}|Q|^2\,\mathrm{d}\theta
\leq
\norm{\omega_h}_{\infty}^2
\frac{1}{2\pi}\int_0^{2\pi}|P|^2\,\mathrm{d}\theta.
\]
Combining this inequality with \eqref{eq:energy-condition} yields
\[
\kappa^2
\frac1{2\pi}\int_0^{2\pi}|P(e^{i\theta})|^2\,\mathrm{d}\theta
\le
\|\omega_h\|_\infty^2
\frac1{2\pi}\int_0^{2\pi}|P(e^{i\theta})|^2\,\mathrm{d}\theta .
\]
Since \(h\) is an orientation-preserving diffeomorphism, \(P\) is not identically zero
on \(|z|=1\). Hence \(\|\omega_h\|_\infty\ge \kappa\).

If $h$ is $K$-quasiconformal, then
\[
\kappa\leq \norm{\omega_h}_{\infty}\leq k=\frac{K-1}{K+1},
\]
which is equivalent to \eqref{eq:normalized-upper}.
\end{proof}

\begin{remark}
The high-frequency counterexample fails this energy test. Its non-first Fourier modes are large
enough on the outer circle to reparametrize the target boundary, but their inner-boundary
contribution is exponentially small in the frequency. Thus, the inner boundary does not see
enough anti-conformal energy to force $\norm{\omega_h}_{\infty}\geq\kappa$.
\end{remark}

\subsection{A Fourier spectral criterion}

The preceding energy criterion becomes explicit when the inner boundary is fixed and the outer
boundary is written in Fourier series. This form also quantifies the leakage mechanism: the
coefficient of a mode of order $n$ is multiplied at the inner boundary by
$2|n|/(s^{|n|}-s^{-|n|})$, which is large for low modes but exponentially small for high modes.

Assume
\begin{equation}\label{eq:fourier-boundary-data}
h(e^{i\theta})=e^{i\theta},
\qquad
h(se^{i\theta})=F(\theta)=\sum_{n\in\mathbb{Z}}c_ne^{in\theta},
\qquad |F(\theta)|=S.
\end{equation}
Set $D=s-s^{-1}$, $a=(c_1-s^{-1})/D$ and $b=(s-c_1)/D$.
Then the first mode is
\[
\left(a\rho+\frac{b}{\rho}\right)e^{i\theta}.
\]
For $n\neq0$, define
\[
\mu_n=\frac{2|n|}{s^{|n|}-s^{-|n|}},
\]
and set $\mu_0=1/\log s$.
Let
\begin{equation}\label{eq:E-definition}
E^2=\mu_0^2|c_0|^2+
\sum_{n\neq0,1}\mu_n^2|c_n|^2.
\end{equation}

\begin{theorem}
\label{thm:fourier-spectral}
Let $S>s$ and set
\[
\kappa=\frac{S-s}{S-s^{-1}}.
\]
Let $h:A(1,s)\to A(1,S)$ be a harmonic orientation-preserving diffeomorphism with boundary
values \eqref{eq:fourier-boundary-data}. If
\begin{equation}\label{eq:spectral-condition}
4|b|^2+E^2\geq \kappa^2\bigl(4|a|^2+E^2\bigr),
\end{equation}
then $\norm{\omega_h}_{\infty}\geq\kappa$.
Consequently, every $K$-quasiconformal map satisfying \eqref{eq:spectral-condition} satisfies the conjectured upper
estimate.
\end{theorem}

\begin{proof}
The harmonic extension of the boundary data has first mode
\[
\left(a\rho+\frac{b}{\rho}\right)e^{i\theta}.
\]
For $n\neq0,1$, the $n$-th mode is
\[
c_nV_{|n|}(\rho)e^{in\theta},
\qquad
V_p(\rho)=\frac{\rho^p-\rho^{-p}}{s^p-s^{-p}},
\]
and the zero mode is
\[
c_0\frac{\log \rho}{\log s}.
\]
On $\rho=1$, the non-first modes themselves vanish, but their normal derivatives do not. A
direct calculation gives $P(e^{i\theta})=2ae^{i\theta}+R(\theta)$ and $Q(e^{i\theta})=-2be^{i\theta}+R(\theta)$,
where
\[
R(\theta)=\mu_0c_0+
\sum_{n\neq0,1}\mu_nc_ne^{in\theta}.
\]
The function $R$ has no $e^{i\theta}$ Fourier mode. By Parseval's identity,
\[
\frac{1}{2\pi}\int_0^{2\pi}|P(e^{i\theta})|^2\,\mathrm{d}\theta
=4|a|^2+E^2,
\]
and
\[
\frac{1}{2\pi}\int_0^{2\pi}|Q(e^{i\theta})|^2\,\mathrm{d}\theta
=4|b|^2+E^2.
\]
Thus, \eqref{eq:spectral-condition} is exactly the energy condition \eqref{eq:energy-condition}. The result follows from the inner-boundary energy criterion.

Equivalently, if $4|b|^2<4\kappa^2|a|^2$, then \eqref{eq:spectral-condition} can be written as
\begin{equation}\label{eq:spectral-threshold}
E^2\geq
\frac{4\bigl(\kappa^2|a|^2-|b|^2\bigr)}{1-\kappa^2}.
\end{equation}
This formulation says that the non-first modes must have sufficient leakage to the inner
boundary. Low modes leak strongly; high modes leak exponentially weakly because
$\mu_n\sim 2|n|s^{-|n|}$.
\end{proof}

\subsection{Small low-frequency boundary reparametrizations}

We now give a concrete family of non-radial maps for which the conjectured estimate holds. Let
$\phi$ be a real-valued trigonometric polynomial with mean zero,
\[
\phi(\theta)=\sum_{\ell\neq0}\phi_\ell e^{i\ell\theta},
\qquad
\phi_{-\ell}=\overline{\phi_\ell},
\qquad
\frac{1}{2\pi}\int_0^{2\pi}\phi(\theta)\,\mathrm{d}\theta=0.
\]
For small $t>0$, prescribe
\begin{equation}\label{eq:lowfreq-boundary-data}
h_t(e^{i\theta})=e^{i\theta},
\qquad
h_t(se^{i\theta})=Se^{i(\theta+t\phi(\theta))}.
\end{equation}
If $t\norm{\phi'}_\infty<1$, the outer boundary map is an orientation-preserving circle
diffeomorphism. Since the corresponding harmonic extension is a $C^1$-small perturbation of
the radial diffeomorphism, $h_t$ is a harmonic orientation-preserving diffeomorphism for all
sufficiently small $t>0$.

Define
\[
\Lambda_{s,S}=\frac{4\kappa(S-s^{-1})}{S(s-s^{-1})^2(1+\kappa)},
\qquad
\kappa=\frac{S-s}{S-s^{-1}}.
\]
Also set
\[
\mu_0=\frac{1}{\log s},
\qquad
\mu_n=\frac{2n}{s^n-s^{-n}},\quad n\geq1.
\]

\begin{theorem}
\label{thm:low-frequency}
Assume $1<s<S$. Let $\phi$ be a real trigonometric polynomial of mean zero. If
\begin{equation}\label{eq:lowfreq-spectral-condition}
\sum_{\ell\neq0}\mu_{|1+\ell|}^2|\phi_\ell|^2
>
\Lambda_{s,S}\sum_{\ell\neq0}|\phi_\ell|^2,
\end{equation}
where $\mu_{|1+\ell|}$ is interpreted as $\mu_0$ when $1+\ell=0$, then for all sufficiently small
$t>0$ the harmonic diffeomorphism $h_t$ defined by \eqref{eq:lowfreq-boundary-data} satisfies
$\norm{\omega_{h_t}}_\infty\geq \kappa(s,S)$.
Consequently, $h_t$ satisfies the conjectured upper estimate.
\end{theorem}

\begin{proof}
Write the outer boundary as
\[
F_t(\theta)=Se^{i\theta}e^{it\phi(\theta)}=
\sum_{n\in\mathbb{Z}}c_n(t)e^{in\theta}.
\]
Since $\langle \phi\rangle=0$,
\[
e^{it\phi}=1+it\phi-\frac{t^2}{2}\phi^2+O(t^3)
\]
in every $C^r$ norm, because $\phi$ is a trigonometric polynomial. Hence
\begin{equation}\label{eq:c1-expansion}
c_1(t)=S\left(1-\frac{t^2}{2}\sum_{\ell\neq0}|\phi_\ell|^2\right)+O(t^3),
\end{equation}
and, for $\ell\neq0$,
$c_{1+\ell}(t)=iSt\phi_\ell+O(t^2)$.
Therefore the leakage quantity in \eqref{eq:E-definition} satisfies
\begin{equation}\label{eq:Et-expansion}
E_t^2=S^2t^2\sum_{\ell\neq0}\mu_{|1+\ell|}^2|\phi_\ell|^2+O(t^3).
\end{equation}
Let
\[
D=s-s^{-1},
\qquad
a(t)=\frac{c_1(t)-s^{-1}}{D},
\qquad
b(t)=\frac{s-c_1(t)}{D}.
\]
The threshold appearing in \eqref{eq:spectral-threshold} is
\begin{equation}\label{eq:T-definition}
T(t)=\frac{4\bigl(\kappa^2|a(t)|^2-|b(t)|^2\bigr)}{1-\kappa^2}.
\end{equation}
At $t=0$, the radial map has
\[
\frac{|b(0)|}{|a(0)|}=\kappa,
\]
so $T(0)=0$. From \eqref{eq:c1-expansion}, write
\[
c_1(t)=S-\delta(t)+O(t^3),
\qquad
\delta(t)=\frac{St^2}{2}\sum_{\ell\neq0}|\phi_\ell|^2.
\]
The possible imaginary part of $c_1(t)$ is $O(t^3)$ and therefore does not affect the leading
$t^2$ term in \eqref{eq:T-definition}. A first-order expansion in $\delta$ gives
\begin{equation}\label{eq:T-expansion}
T(t)=\frac{4S\kappa(S-s^{-1})}{(s-s^{-1})^2(1+\kappa)}
t^2\sum_{\ell\neq0}|\phi_\ell|^2+O(t^3).
\end{equation}
The strict spectral condition \eqref{eq:lowfreq-spectral-condition} is exactly the statement that the leading coefficient in \eqref{eq:Et-expansion} is larger than the leading coefficient in \eqref{eq:T-expansion}. Therefore, for all sufficiently small $t>0$, $E_t^2>T(t)$.
By the Fourier spectral criterion, this implies $\norm{\omega_{h_t}}_\infty\geq\kappa(s,S)$.
The equivalence \eqref{eq:kappa-bound} then gives the conjectured upper estimate for every $K$ for which $h_t$ is
$K$-quasiconformal.
\end{proof}

\begin{example}[The pair $(s,S)=(2,3)$]
Here
\[
\kappa(2,3)=\frac25,
\qquad
\Lambda_{2,3}=\frac{4\cdot(2/5)\cdot(3-1/2)}{3(2-1/2)^2(1+2/5)}\approx0.423280.
\]
Moreover
\[
\mu_0=\frac{1}{\log2}\approx1.4427,
\qquad
\mu_n=\frac{2n}{2^n-2^{-n}}.
\]
For the single-frequency perturbation $\phi(\theta)=\sin m\theta$, condition \eqref{eq:lowfreq-spectral-condition} becomes
\[
\frac12\left(\mu_{m+1}^2+\mu_{|m-1|}^2\right)>0.423280.
\]
The first few values are
\[
\begin{array}{c|c|c}
 m & \frac12(\mu_{m+1}^2+\mu_{|m-1|}^2) & \text{conclusion} \\
\hline
1 & 1.6096\ldots & \text{stable} \\
2 & 1.1791\ldots & \text{stable} \\
3 & 0.6949\ldots & \text{stable} \\
4 & 0.3392\ldots & \text{not guaranteed by this criterion}
\end{array}
\]
Thus, for $m=1,2,3$ and all sufficiently small $t>0$, the non-radial harmonic maps
\[
h_t(e^{i\theta})=e^{i\theta},
\qquad
h_t(2e^{i\theta})=3e^{i(\theta+t\sin m\theta)}
\]
satisfy the conjectured upper estimate. The high-frequency counterexample with $m=20$ lies on
the opposite side of this mechanism: the relevant $\mu_{19}$ and $\mu_{21}$ are exponentially
small, so the non-first modes do not leak enough energy to the inner boundary.
\end{example}

\section{Minimal-surface interpretation}\label{sec:minimal}

The results above have a useful geometric interpretation through the classical correspondence
between harmonic maps and conformal minimal immersions. The purpose of this section is not to
reprove this correspondence, but to identify precisely which minimal-surface category is affected
by the counterexamples.

There are two different regimes. In the single-valued regime, the height function has zero vertical
period around the annulus. In the helicoidal, or vertical-periodic, regime, the height is defined on
the universal cover and changes by a non-zero constant after one turn. The counterexamples
constructed in Section~\ref{sec3} naturally belong to the second regime.

\subsection{The standard dictionary}

Let
\[
h=f+\overline g:A(1,s)\longrightarrow \C
\]
be an orientation-preserving harmonic local diffeomorphism. We shall use the following standard
facts from the harmonic-map/minimal-surface correspondence, in the normalization needed below.

First, if there is a holomorphic function $p$ on $A(1,s)$ such that
\begin{equation}\label{eq:p-square}
p(z)^2=-f'(z)g'(z),
\end{equation}
then, on the universal cover of the annulus, after fixing a base point $z_0$ and lifting the
one-form $p(z)\,dz$ to that cover,
\[
w(z)=2\operatorname{Re}\int_{z_0}^{z} p(\zeta)\,\mathrm{d}\zeta
\]
defines a conformal minimal immersion
\[
X=(\operatorname{Re}h,\operatorname{Im}h,w).
\]
The height descends to a single-valued function on the annulus if and only if
\begin{equation}\label{eq:zero-period}
2\operatorname{Re}\int_\gamma p(z)\,\mathrm{d}z=0,
\end{equation}
where $\gamma$ is a positively oriented generator of $H_1(A(1,s),\mathbb Z)$. If this real period
is non-zero, then the lift is vertical-periodic and descends to a minimal annulus in
\[
\mathbb R^2\times \mathbb R/T\mathbb Z,
\qquad
T=2\operatorname{Re}\int_\gamma p(z)\,\mathrm{d}z.
\]

Second, suppose that $X=(h,w)$ is a minimal graph, or a vertical-periodic minimal multigraph,
and that its horizontal projection $h:A(1,s)\to A(1,S)$
is an orientation-preserving diffeomorphism. Let
\[
W=w\circ h^{-1}
\]
be the height over the horizontal annulus, interpreted on the universal cover in the periodic case.
Then the complex dilatation of the horizontal projection and the slope of the minimal graph satisfy
\begin{equation}\label{eq:slope-dilatation}
|\nabla W|\circ h=\frac{2\sqrt{|\omega_h|}}{1-|\omega_h|},
\end{equation}
and equivalently
\begin{equation}\label{eq:K-slope}
\sqrt{1+|\nabla W|^2}\circ h
=
\frac{1+|\omega_h|}{1-|\omega_h|}.
\end{equation}

Consequently, in the thickening range $S>s$, the planar threshold
\[
\kappa(s,S)=\frac{S-s}{S-s^{-1}}
\]
corresponds to the minimal-surface slope threshold
\begin{equation}\label{eq:L-kappa}
L_\kappa(s,S)=\frac{2\sqrt{\kappa(s,S)}}{1-\kappa(s,S)}.
\end{equation}

\subsection{The one-mode sign and the vertical period}

The one-mode maps show why the upper radial model is helicoidal rather than catenoidal. Consider
\[
h(z)=az+\frac{b}{\overline z},\qquad a,b\in\C.
\]
Then $f(z)=az$ and $g(z)=\overline b/z$, so
\[
f'(z)=a,\qquad g'(z)=-\frac{\overline b}{z^2},
\]
and the lift equation becomes
\begin{equation}\label{eq:one-mode-p}
p(z)^2=-f'(z)g'(z)=\frac{a\overline b}{z^2}.
\end{equation}
Thus $p(z)=C/z$, where $C^2=a\overline b$, and the vertical period is
\[
T=2\operatorname{Re}\int_{|z|=1}\frac{C}{z}\,\mathrm{d}z
  =2\operatorname{Re}(2\pi iC)
  =-4\pi\operatorname{Im}C.
\]

It follows immediately that a single-valued one-mode lift requires $C$ to be real, hence $a\overline b\in[0,\infty)$.
After a rotation of the target, the single-valued one-mode case has the form
\[
h(z)=A\left(z+\frac{q}{\overline z}\right),
\qquad A>0,\qquad 0\leq q<1.
\]
With the normalization
\[
|a+b|=1,
\qquad
\left|as+\frac{b}{s}\right|=S,
\]
this gives
\[
S=\frac{s+q/s}{1+q}\leq s.
\]
Thus the catenoidal sign does not produce the thickening case $S>s$.

By contrast, the upper radial model has the opposite sign,
\[
h(z)=A\left(z-\frac{q}{\overline z}\right),
\qquad A>0,\qquad q>0.
\]
Here $a\overline b=-A^2q<0$,
so $C=\pm iA\sqrt q$, and
\[
T=\mp 4\pi A\sqrt q\ne0.
\]
Hence the upper radial model is not a single-valued catenoidal annulus. It is helicoidal, or
vertical-periodic.

This sign distinction is the geometric reason why the planar counterexamples below affect the
vertical-periodic category, but do not automatically settle the zero-period single-valued graph
problem.

\subsection{The counterexamples as vertical-periodic minimal annuli}

The harmonic counterexamples from Theorem~\ref{thm:counterexamples} therefore have a natural
minimal-surface interpretation, but in the periodic category.

\begin{theorem}\label{thm:minimal-counter}
For every $1<s<S$, there exists a conformal minimal immersion
\[
\widetilde X:\widetilde{A(1,s)}\to\mathbb R^3
\]
on the universal cover of the annulus such that:
\begin{enumerate}[label=\textup{(\roman*)}]
\item its horizontal projection is a harmonic diffeomorphism
\[
h:A(1,s)\to A(1,S);
\]
\item it has non-zero vertical period;
\item it is a vertical-periodic minimal multigraph over $A(1,S)$; and
\item its maximal slope satisfies
\[
\sup |\nabla W|
<
\frac{2\sqrt{\kappa(s,S)}}{1-\kappa(s,S)}.
\]
\end{enumerate}
Thus the planar counterexamples disprove the corresponding upper slope principle in the
vertical-periodic helicoidal category.
\end{theorem}

\begin{proof}
Let $h_t$ be the maps constructed in the proof of Theorem~\ref{thm:counterexamples}. As $t\to0$,
they converge in $C^1$ on the closed annulus to the upper radial one-mode map
\[
h_0(z)=a_0z-\frac{\beta_0}{\overline z},
\qquad
a_0=\frac{S-s^{-1}}{s-s^{-1}},
\qquad
\beta_0=\frac{S-s}{s-s^{-1}}>0.
\]
For this limiting map, $f_0(z)=a_0z$ and $g_0(z)=-\beta_0/z$. Hence
\[
p_0(z)^2=-f_0'(z)g_0'(z)=-\frac{a_0\beta_0}{z^2},
\qquad
p_0(z)=\frac{i\sqrt{a_0\beta_0}}{z}
\]
up to sign. Let $\gamma$ be any positively oriented circle $|z|=r_0$ with $1<r_0<s$. The corresponding
vertical period is
\[
T_0=2\operatorname{Re}\int_{\gamma}p_0(z)\,\mathrm{d}z
   =-4\pi\sqrt{a_0\beta_0}
\]
up to sign, and hence is non-zero.

We next check that the square-root condition persists under the perturbation. The estimates in
the proof of Theorem~\ref{thm:counterexamples} give, uniformly on the closed annulus,
\[
2zf_t'(z)=2a_0z+O(m(t)t).
\]
The same calculation applied to
$Q_t=\rho(h_t)_\rho+i(h_t)_\theta=\overline{2zg_t'(z)}$ gives
\[
2zg_t'(z)=\frac{2\beta_0}{z}+O(m(t)t).
\]
Since $a_0,\beta_0>0$ and $m(t)t\to0$, both $f_t'$ and $g_t'$ are non-vanishing on the closed
annulus for all sufficiently small $t$. Consequently $-f_t'g_t'$ is a non-vanishing holomorphic
function which is uniformly close to $-a_0\beta_0/z^2$ on the closed annulus. Its winding number
around a generator is therefore $-2$, the same as that of $-a_0\beta_0/z^2$. We use here the
standard criterion that a non-vanishing holomorphic function on an annulus has a holomorphic
square root if and only if its winding number along a generator is even. Hence $-f_t'g_t'$ admits a
holomorphic square root $p_t$ on $A(1,s)$. Choosing the branch continuously from $p_0$ gives
$p_t\to p_0$ uniformly on compact subannuli, in particular along $\gamma$. The associated minimal
lift has period
\[
T_t=2\operatorname{Re}\int_{\gamma}p_t(z)\,\mathrm{d}z,
\]
and $T_t\to T_0$. Thus $T_t\ne0$ for all sufficiently small $t$.

Finally, Theorem~\ref{thm:counterexamples} gives
\[
\|\omega_{h_t}\|_\infty<\kappa(s,S).
\]
Using \eqref{eq:slope-dilatation} and the monotonicity of
\[
q\mapsto \frac{2\sqrt q}{1-q}
\]
on $(0,1)$, we obtain
\[
\sup |\nabla W_t|
=
\frac{2\sqrt{\|\omega_{h_t}\|_\infty}}
     {1-\|\omega_{h_t}\|_\infty}
<
\frac{2\sqrt{\kappa(s,S)}}{1-\kappa(s,S)}.
\]
\end{proof}

\begin{remark}
The surfaces in Theorem~\ref{thm:minimal-counter} should not be interpreted as minimizing
surfaces. They show only that the radial helicoidal threshold is not sharp in the unrestricted
vertical-periodic category. The actual minimizer of $\|\omega_h\|_\infty$, equivalently the
minimizer of the maximal slope through \eqref{eq:slope-dilatation}, is expected to be a different
non-radial object.
\end{remark}

\subsection{Positive planar criteria as slope criteria}

The positive criteria in Section~\ref{sec4} immediately become lower bounds for the slope of
minimal graphs or vertical-periodic multigraphs.

\begin{corollary}\label{cor:minimal-positive}
Let $X=(h,w)$ be a minimal graph or a vertical-periodic minimal multigraph whose horizontal
projection
\[
h:A(1,s)\to A(1,S),
\qquad S>s,
\]
is a harmonic orientation-preserving diffeomorphism. Suppose that $h$ satisfies one of the
following hypotheses:
\begin{enumerate}[label=\textup{(\alph*)}]
\item $h$ is one-mode of the form covered by Lemma~\ref{lem:onemode};
\item $h$ satisfies the inner-boundary anti-conformal energy condition
      \eqref{eq:energy-condition};
\item $h$ satisfies the Fourier spectral condition \eqref{eq:spectral-condition}; or
\item $h$ is a sufficiently small low-frequency boundary reparametrization satisfying
      \eqref{eq:lowfreq-spectral-condition}.
\end{enumerate}
Then
\[
\sup |\nabla W|
\geq
\frac{2\sqrt{\kappa(s,S)}}{1-\kappa(s,S)}.
\]
Equivalently, if $\sup |\nabla W|\leq L$, then
\[
S\leq \frac{K_L+1}{2}s-\frac{K_L-1}{2s},
\qquad
K_L=\sqrt{1+L^2}.
\]
\end{corollary}

\begin{proof}
Each of the hypotheses listed above implies
\[
\|\omega_h\|_\infty\geq \kappa(s,S)
\]
by the corresponding planar result in Section~\ref{sec4}. The slope bound follows from
\eqref{eq:slope-dilatation}.

Conversely, if $\sup |\nabla W|\leq L$, then \eqref{eq:K-slope} gives
\[
\|\omega_h\|_\infty
\leq
\frac{K_L-1}{K_L+1},
\qquad
K_L=\sqrt{1+L^2}.
\]
Combining this with
\[
\kappa(s,S)\leq \|\omega_h\|_\infty
\]
and rearranging the definition of $\kappa(s,S)$ gives the stated upper estimate for $S$.
\end{proof}

\subsection{Local zero-period evidence}

The previous counterexamples are periodic. The zero-period condition for a single-valued minimal
graph removes the helicoidal extremal direction at the infinitesimal level. The following result is
local near the flat annulus and should not be read as a solution of the global single-valued problem.

Let \(D=A(1,S)\), and let \(u\in C^{2,\alpha}(\overline D)\) be a single-valued harmonic
function. For \(|\varepsilon|\) small, let \(W_\varepsilon\) be the solution of the minimal graph
equation on \(D\) with boundary values
\[
W_\varepsilon|_{\partial D}=\varepsilon u|_{\partial D}.
\]
Then
\[
W_\varepsilon=\varepsilon u+O(\varepsilon^3)
\quad\text{in }C^{2,\alpha}(\overline D).
\]
The absence of an $\varepsilon^2$ term follows from the oddness of the minimal graph equation
under $W\mapsto -W$, together with the standard implicit-function and Schauder theory for the
Dirichlet problem.
Let \(\log s_\varepsilon\) be the conformal modulus of the graph of \(W_\varepsilon\). In conformal
coordinates, the horizontal projection is a harmonic diffeomorphism $h_\varepsilon:A(1,s_\varepsilon)\to A(1,S)$.

Define
\begin{equation}\label{eq:Iu}
I(u)=\int_1^S\int_0^{2\pi}
\frac{u_\theta^2/r^2-u_r^2}{r}\,\mathrm{d}\theta\,\mathrm{d}r,
\end{equation}
and
\[
M(u)=\sup_{A(1,S)}|\nabla u|^2.
\]
Let \(F=2u_z\), and write
\[
F^2=\sum_{n\in\mathbb Z}\alpha_n z^n.
\]
Set $C_2(u):=-\operatorname{Re}\alpha_{-2}$.
Then
\[
I(u)=\pi(1-S^{-2})C_2(u).
\]
Thus the condition \(I(u)>0\) is equivalent to $C_2(u)>0$.

\begin{theorem}\label{thm:zero-period-perturbative}
Assume that $I(u)>0$. Then $s_\varepsilon<S$ for all sufficiently small positive $\varepsilon$, and
\[
\|\omega_{h_\varepsilon}\|_\infty>
\kappa(s_\varepsilon,S)
=
\frac{S-s_\varepsilon}{S-s_\varepsilon^{-1}}
\]
for all sufficiently small positive $\varepsilon$.
\end{theorem}

\begin{proof}
Let \(L_0=\log S\). The graph metric over \(D=A(1,S)\) is
\[
g_\varepsilon
=
|dw|^2+dW_\varepsilon^2
=
g_{\mathrm{euc}}+\varepsilon^2\,du^2+O(\varepsilon^3).
\]
The conformal modulus is determined by capacity:
\[
\operatorname{Cap}(D,g_\varepsilon)=\frac{2\pi}{\log s_\varepsilon}.
\]
For the Euclidean annulus, the extremal function is
\[
v_0(r)=\frac{\log r}{L_0}.
\]
Expanding the Dirichlet integrand gives
\[
|\nabla_{g_\varepsilon}v|^2\,dA_{g_\varepsilon}
=
\left(
|\nabla v|^2
+
\varepsilon^2
\left[
\frac12|\nabla u|^2|\nabla v|^2
-
(\nabla u\cdot\nabla v)^2
\right]
\right)\mathrm{d}A
+
O(\varepsilon^3).
\]
Since \(v_0\) is the Euclidean minimizer, the first non-trivial variation of capacity is obtained by
inserting \(v_0\). Thus
\[
\operatorname{Cap}(D,g_\varepsilon)
=
\frac{2\pi}{L_0}
+
\frac{\varepsilon^2}{2L_0^2}I(u)
+
O(\varepsilon^3).
\]
Consequently,
\begin{equation}\label{eq:mod-expansion}
\log s_\varepsilon
=
L_0-\frac{\varepsilon^2}{4\pi}I(u)+O(\varepsilon^3).
\end{equation}
In particular, \(I(u)>0\) implies \(s_\varepsilon<S\) for all sufficiently small positive
\(\varepsilon\).

From \eqref{eq:mod-expansion},
\[
\kappa(s_\varepsilon,S)
=
\frac{S^2}{S^2-1}\frac{I(u)}{4\pi}\,\varepsilon^2
+
O(\varepsilon^3).
\]
On the other hand, the slope formula gives
\[
\|\omega_{h_\varepsilon}\|_\infty
=
\frac{\varepsilon^2}{4}M(u)
+
O(\varepsilon^3).
\]
It remains to compare \(I(u)\) and \(M(u)\).

Put $F=2u_z=u_x-iu_y$.
Then \(F\) is single-valued and holomorphic in \(D\), and
\[
|F|^2=|\nabla u|^2.
\]
In polar coordinates,
\[
F=e^{-i\theta}\left(u_r-\frac{i}{r}u_\theta\right),
\]
so
\[
\frac{u_\theta^2}{r^2}-u_r^2
=
-\operatorname{Re}\left(e^{2i\theta}F^2\right).
\]
Write the Laurent expansion
\[
F^2=\sum_{n\in\mathbb Z}\alpha_n z^n.
\]
Only the coefficient \(\alpha_{-2}\) survives after integration in \(\theta\), and hence
\begin{equation}\label{eq:I-cminus2}
I(u)
=
-\pi(1-S^{-2})\operatorname{Re}\alpha_{-2}.
\end{equation}
Moreover,
\[
\alpha_{-2}
=
\frac{1}{2\pi i}\int_{|z|=1}F(z)^2z\,\mathrm{d}z
=
\frac1{2\pi}\int_0^{2\pi}
e^{2i\theta}F(e^{i\theta})^2\,\mathrm{d}\theta.
\]
Therefore
\[
|\alpha_{-2}|
\leq
\frac1{2\pi}\int_0^{2\pi}|F(e^{i\theta})|^2\,\mathrm{d}\theta
\leq
M(u),
\]
and so
\begin{equation}\label{eq:I-M}
I(u)\leq \pi(1-S^{-2})M(u).
\end{equation}

For \(I(u)>0\), equality in \eqref{eq:I-M} is impossible. Indeed, equality in
\[
I(u)\leq \pi(1-S^{-2})|\alpha_{-2}|
\leq \pi(1-S^{-2})M(u)
\]
would force equality in both inequalities. Hence $e^{2i\theta}F(e^{i\theta})^2$
has constant argument on \(|z|=1\), and $|F(e^{i\theta})|^2=M(u)$
there. Therefore \(e^{i\theta}F(e^{i\theta})\) is constant on \(|z|=1\), say
$e^{i\theta}F(e^{i\theta})=C$.
Thus, by uniqueness of Laurent coefficients,
\[
F(z)=\frac{C}{z}.
\]
Consequently,
\[
u=\operatorname{Re}(C\log z)+\mathrm{constant}.
\]
Since \(u\) is single-valued, \(C\) must be real. But then
$\alpha_{-2}=C^2\geq0$,
and so
\[
I(u)=-\pi(1-S^{-2})C^2\leq0,
\]
contradicting \(I(u)>0\). Hence the inequality in \eqref{eq:I-M} is strict whenever \(I(u)>0\).

The strict form of \eqref{eq:I-M} gives
\[
\frac{S^2}{S^2-1}\frac{I(u)}{4\pi}
<
\frac14M(u).
\]
Comparing the two expansions above yields
\[
\kappa(s_\varepsilon,S)
<
\|\omega_{h_\varepsilon}\|_\infty
\]
for all sufficiently small positive \(\varepsilon\).
\end{proof}

\subsection{The remaining single-valued problem}

The counterexamples above are vertical-periodic. The corresponding zero-period, single-valued
minimal-graph problem remains separate: Theorem~\ref{thm:minimal-counter} does not answer it,
because its height function has non-zero vertical period.

\begin{problem}\label{prob:zero-period}
Let \(W:A(1,S)\to\mathbb R\) be a single-valued minimal graph with
\[
\sup_{A(1,S)}|\nabla W|\le L,
\]
and let \(\log s\) be the conformal modulus of its graph. If \(S>s\), must
\[
S\le \frac{K+1}{2}s-\frac{K-1}{2s},
\qquad K=\sqrt{1+L^2},
\]
or equivalently
\[
\log s\ge \int_1^S\frac{\mathrm{d}r}{\sqrt{r^2+L^2}},
\]
hold?
\end{problem}

The comparison metric is the helicoidal metric
\[
dr^2+(r^2+L^2)\mathrm{d}\theta^2,
\]
whose extremal height is \(W=L\theta\). Since this height has non-zero period, it is not
single-valued. Thus the question is whether the zero-period condition rules out the helicoidal
escape mechanism globally.

Equivalently, in conformal parameters, the problem asks for the same sharp bound for harmonic
annulus diffeomorphisms
\[
F:A(1,s)\to A(1,S),\qquad K_F\le K,
\]
under the additional minimal-graph condition
\[
\Phi_F=F_z\overline{F_{\bar z}}\,dz^2=-\varphi^2,
\]
where \(\varphi\) is a holomorphic one-form with zero real periods. This zero-period condition is
exactly the single-valuedness of the height function.

\section{Concluding interpretation}\label{sec:conclusion}

The results of this paper show that the upper Nitsche--Gr\"otzsch principle proposed in
Conjecture~\ref{conj:iko} is a genuinely one-mode phenomenon. The radial model, the
spiral-radial model, and the principal harmonic maps in \cite{IKO2012} all lead to the
same sharp quantity
\[
\frac{K+1}{2}s-\frac{K-1}{2s}.
\]
However, Theorem~\ref{thm:counterexamples} shows that this quantity does not control
arbitrary non-radial harmonic annulus diffeomorphisms.

The mechanism of failure is frequency filtering in the annulus. A Fourier mode of order
\(n\) prescribed on the outer boundary reaches the inner boundary with size comparable to
\[
\frac{2|n|}{s^{|n|}-s^{-|n|}}\sim 2|n|s^{-|n|}.
\]
Thus, high frequencies are exponentially damped at the inner boundary. A small
high-frequency angular reparametrization can lower the first Fourier-mode dilatation by
order \(t^2\), while the compensating non-first modes remain \(o(t^2)\) there when
\(m\asymp \log(1/t)\). This is why the radial dilatation threshold can be beaten in the
full non-radial class.

The positive results in Section~\ref{sec4} identify the regimes in which this escape is
unavailable. If enough anti-conformal energy is visible on the inner boundary, or
equivalently if the non-first Fourier modes have sufficient inner-boundary leakage, then
the one-mode estimate survives. In particular, low-frequency perturbations satisfying the
spectral condition remain governed by the conjectured bound.

The minimal-surface interpretation separates two different geometric problems. The upper
radial model has helicoidal sign and carries non-zero vertical period, whereas the
catenoidal sign is compatible with a single-valued one-mode minimal annulus. Consequently,
the counterexamples disprove both the planar harmonic conjecture and the corresponding
vertical-periodic helicoidal slope principle, but they do not by themselves disprove the
zero-period problem for single-valued minimal graphs.

A further consequence is that the radial candidate is not extremal for the unrestricted
harmonic quasiconformal problem. In the notation of
Corollary~\ref{cor:nonradial-extremal-value},
\[
q_*(s,S)<\kappa(s,S)=\frac{S-s}{S-s^{-1}},
\]
whereas the radial helicoidal model has dilatation exactly \(\kappa(s,S)\). Thus the
natural replacement for Conjecture~\ref{conj:iko} is not a corrected radial estimate, but
the following non-radial extremal problem.

\begin{problem}\label{prob:extremal}
For \(1<s<S\), determine
\[
q_*(s,S)=\inf_h \|\omega_h\|_\infty,
\]
where the infimum is taken over all harmonic orientation-preserving quasiconformal
diffeomorphisms \(h:A(1,s)\to A(1,S)\). Describe the extremal maps, if they exist. Are
they unique up to rotations of the source and target? What are their boundary
reparametrizations, and what is the geometry of their minimal-surface lifts?
\end{problem}

The high-frequency examples constructed here show only that the radial value is not
optimal.

\appendix
\section{Numerical details for the explicit \texorpdfstring{$(2,3)$}{(2,3)} counterexample}\label{Apx}

For completeness we record the estimates leading to the explicit bound
\[
\norm{\omega_h}_{\infty}<0.39993<\frac25
\]
when $s=2$, $S=3$, $t=1/50$, and $m=20$.

The outer boundary is
\[
h(2e^{i\theta})=3e^{i(\theta+t\sin20\theta)}=
\sum_{j\in\mathbb{Z}}c_je^{in_j\theta},
\]
where
\[
c_j=3J_j(t),
\qquad
n_j=1+20j,
\qquad
p_j=|n_j|.
\]
The first-mode coefficients are
\[
a=\frac{3J_0(t)-1/2}{3/2},
\qquad
b=\frac{2-3J_0(t)}{3/2},
\qquad
\beta=-b.
\]
Since
\[
J_0(t)=1-\frac{t^2}{4}+\frac{t^4}{64}-\cdots,
\qquad
\text{where } t=\frac{1}{50},
\]
we have
\[
1-10^{-4}<J_0(t)<1-10^{-4}+\frac14\cdot10^{-8},
\]
and hence
\[
a>1.66646666,
\qquad
\beta<0.666466672.
\]
The quantities $R_1$, $R_2^P$, and $R_2^Q$ below are the inner-boundary perturbation bound,
the outer-boundary $P$-perturbation bound, and the outer-boundary $Q$-perturbation bound from
the proof of Theorem~\ref{thm:counterexamples}, specialized to these parameters.
For $j\geq1$,
\[
|J_j(t)|\leq e^{10^{-4}}\frac{10^{-2j}}{j!},
\qquad
|J_{-j}(t)|=|J_j(t)|.
\]
At the inner boundary,
\[
R_1=\sum_{j\neq0}|c_j|\frac{2p_j}{2^{p_j}-2^{-p_j}}<2.78\cdot10^{-6}.
\]
At the outer boundary, with
\[
\lambda_p=p\frac{2^p+2^{-p}}{2^p-2^{-p}},
\]
one obtains
\[
R_2^P=\sum_{j\neq0}|c_j|\,|\lambda_{p_j}+n_j|<1.273,
\qquad
R_2^Q=\sum_{j\neq0}|c_j|\,|\lambda_{p_j}-n_j|<1.152.
\]
The same estimate gives an interior bound for the non-first contribution to $P=2zh_z$:
\[
\sum_{j\neq0}|c_j|\sup_{1\leq r\leq2}
\left|rV_{p_j}'(r)+n_jV_{p_j}(r)\right|<1.273.
\]
Therefore
\[
|P|\geq 2a-1.273>2
\]
on the closed annulus, and thus $\omega_h$ is holomorphic. On $|z|=1$,
\[
|\omega_h|\leq \frac{2\beta+R_1}{2a-R_1}<0.39993,
\]
while on $|z|=2$,
\[
|\omega_h|\leq \frac{\beta+R_2^Q}{4a-R_2^P}<0.338.
\]
The maximum principle gives the desired numerical estimate.

\medskip

\noindent\textbf{Acknowledgements.} We are grateful to Leonid Kovalev for his careful reading of an earlier version, for his encouraging
comments on the counterexample, and for pointing out the natural extremal problem for harmonic
quasiconformal diffeomorphisms between circular annuli.

\medskip

\noindent\textbf{Funding.} The first author gratefully acknowledges financial support from the Ministry of Education, Science and Innovation of Montenegro through the grants \emph{``Mathematical Analysis, Optimisation and Machine Learning''} and \emph{``Complex-analytic and geometric techniques for non-Euclidean machine learning: theory and applications.''} The second author was supported by National Key R\&D Program of China (Grant No. 2021YFA1003100) and NSF of Guangdong Province (Grant No. 2025A1515011213). The third author was supported by NSF of China (No. 12271189, 12671096), NSF of Guangdong Province (Grant No. 2024A1515010467, 2026A1515012333), STU Scientific Research Initiation Grant NTF25017T, HQU teaching reform project HQJGKT2411, and Fujian Alliance of Mathematics (Grant No. 2023SXLMMS07).

\medskip

\noindent\textbf{Data availability.} The authors declare that this research is purely theoretical and does not involve any data.

\medskip

\noindent\textbf{Conflicts of interest.} The authors declare that they have no conflicts of interest regarding the publication of this paper.

\end{document}